\documentclass[a4paper,10pt]{article}
\usepackage{amsmath}
\usepackage{amssymb}
\usepackage{latexsym}
\usepackage{amsthm}
\usepackage{tikz,
}
\allowdisplaybreaks

\newtheorem{thm}{Theorem}[section]

\newtheorem{lem}[thm]{Lemma}
\newtheorem{prop}[thm]{Proposition}

\newtheorem{constr}[thm]{Construction}

\theoremstyle{definition}
\newtheorem{defn}[thm]{Definition}

\theoremstyle{remark}

\numberwithin{equation}{section}

\newcommand{\Z}{\mathbb{Z}}

\newcommand{\ol}{\overline}
\newcommand{\oDelta}{\overrightarrow{\Delta}}

\begin{document}
\normalsize
\date{}
\title{On the existence of unparalleled even cycle systems}
\author{
Peter Danziger
\thanks{Department of Mathematics, Ryerson University, Toronto, ON M5B 2K3, Canada}
\textsuperscript{,\hspace{-.1cm}}
\thanks{The author is supported by an NSERC Discovery grant. 
E-mail: danziger@ryerson.ca}
\and
Eric Mendelsohn
\footnotemark[1]
\textsuperscript{,\hspace{-.1cm}}
\thanks{E-mail: mendelso@math.utoronto.ca}
\and
Tommaso Traetta
\footnotemark[1]
\textsuperscript{,\hspace{-.1cm}}
\thanks{The author is supported by a fellowship of INdAM. 
E-mail: tommaso.traetta@ryerson.ca, traetta.tommaso@gmail.com}
}

\maketitle
\begin{abstract} 
\noindent A $2t$-cycle system of order $v$ 
is a set $\mathcal{C}$ of cycles whose edges partition the edge-set of $K_v-I$ (i.e., the complete graph minus the $1$-factor $I$). If $v\equiv 0 \pmod{2t}$, a set of $v/2t$ vertex-disjoint cycles of $\mathcal{C}$ 
is a parallel class. If $\mathcal{C}$  has no parallel classes, we call such a system {unparalleled}.

We show that there exists an unparalleled $2t$-cycle system of order $v \equiv 0 \pmod{2t}$ if and only if
$v>2t>2$. 
\end{abstract}

\noindent {\bf Keywords: cycle system, resolvability, parallel class free, $1$-rotational, $2$-pyramidal,
semiregular}
\eject

\section{Introduction}
A \emph{cycle system} of a simple graph $\Gamma$ is a set $\mathcal{C}$ of cycles of $\Gamma$ whose edge-sets 
partition the edge-set $E(\Gamma)$ of $\Gamma$. If $\Gamma$ contains only cycles of length $k$ 
(briefly, \emph{$k$-cycles}), we speak of a \emph{$k$-cycle system}. 
A set of cycles of $\mathcal{C}$ whose vertex-sets partition the vertex-set $V(\Gamma)$ of $\Gamma$ is 
called a \emph{parallel class} of $\mathcal{C}$. If the set $\mathcal{C}$ of cycles can be partitioned into parallel classes 
the system is called \emph{resolvable}. If $\mathcal{C}$ has no parallel classes, 
we call such a system \emph{unparalleled}.

We denote by $K_V$ and $K_v$ the \emph{complete graphs} with vertex-set $V$ and order $v$, respectively.
When the graph $\Gamma$ is $K_v$ or $K_v - I$ (i.e., $K_v$ minus the $1$-factor $I$) 
we refer to a $k$-cycle system of order $v$ and denote it with $CS(v,k)$. It is known that a $CS(v,k)$ exists if and only if $v\geq k\geq3$ and $v\lfloor\frac{v-1}{2} \rfloor \equiv 0 \pmod{k}$, \cite{AlGa01, Sa02}.

This paper deals with the existence of unparalleled $k$--cycle systems of order $v$. 
Since a $CS(v,k)$ contains a parallel class only if 
$v\equiv 0 \pmod{k}$, we will always assume that $v \equiv 0 \pmod{k}$, otherwise the existence of an unparalleled $CS(v,k)$ is straightforward. Also, note that a $CS(v,v)$ is always resolvable hence, we will 
assume that $v > k$.

Here, we solve the existence problem of an unparalleled even cycle system of order $v$.
More precisely, we show what follows:
\begin{thm}\label{main} 
There exists an unparalleled $CS(v,2t)$ with $v \equiv 0 \pmod{2t}$ if and only if
$v>2t>2$. 
\end{thm}

We point out that the equivalent problem concerning the existence of a resolvable $CS(v,k)$ has been completely settled in a series of papers
\cite{ASSW89, HoSc91, Lu92, RCW71, ReSt87} showing that:
\begin{itemize}
\item there exists a resolvable $CS(v,k)$ if and only if $v\equiv 0 \pmod{k}$. 
\end{itemize}
Resolvability in cycle systems has been widely studied, but
there are still open questions.
The \emph{Oberwolfach problem} and the 
\emph{Hamilton-Waterloo problem} are two well known examples of open problems about resolvable cycle systems. 
Some recent results on these topics can be found in \cite{BrDa11, BrSch09, BuDaTr, Tr13}. 
For the older ones, we refer the reader to \cite{BrRo07}. We also refer to \cite{BBRT14, BLT14, BuRiTr13}
for some recent progress on problems concerning the automorphisms of a resolvable cycle system. 
 
The existence problem of unparalleled $CS(v,k)$ is opposite 
to resolvability in cycle systems and very little is known. 
For $k=3$, a longstanding
conjecture, dating back to at least 1984, claims that $CS(v,3)$ without parallel classes exist whenever
$v\equiv 3 \pmod{6}$ and $v>9$, (see \cite{CoRo99}). 
The second author believes he was the first  to make this conjecture in conversations with K. Phelps, C.C. Lindner and A. Rosa circa 1984 but would welcome any earlier occurrence.
This conjecture is still open, but an important step forward has been recently taken in \cite{BrHo15} (see also \cite{BrHo13}). As far as we are aware, nothing is known when $k>3$. However, some work has been done on systems which are not resolvable. In particular, 
it is known that there exists a non-resolvable $CS(v,3)$ for all permissible $v\geq 9$, \cite{LiRees05}.

In Section 2 we describe the techniques used in this paper. We will need to deal with 
the \emph{complete bipartite graph} denoted either by 
$K_{X,Y}$ or $K_{r,s}$ whose \emph{parts}
are $X$ and $Y$ or have size $r$ and $s$, respectively.
Also, by 
$CS(K_{r,s}, k)$ we will denote a $k$-cycle system of $K_{r,s}$. 

In Section 3 we prove Theorem \ref{main} when $t=2$.
In Sections 4 and 5 we deal with what we call \emph{intersecting} cycle systems defined as follows:
\begin{enumerate}
 \item[-] a $CS(v,k)$ is  intersecting if any two of its cycles intersect in at least one vertex;
 \item[-] a $CS(K_{X,Y},2t)$ is  {intersecting} if any two of its cycles intersect in at least one vertex of $\mathbf X$.
\end{enumerate} 
Note that an intersecting cycle system is necessarily free from parallel classes. This stronger property is used in Section 6 to prove Theorem \ref{main} for $t>2$ by using a recursive construction (Construction \ref{constr}) that makes use of an intersecting CS$(v,2t)$ for $v=4t$ or $6t$ (Section \ref{intersecting1}) for the base case and isomorphic copies of an intersecting $CS(K_{w,w}, 2t)$ for $w=2t$ or $4t$ (Section \ref{intersecting2}).


\section{Some preliminaries}
We collect here the necessary preliminary notation and definitions, and describe the techniques used in Sections
3, 4, and 5. We refer the reader to \cite{Sc87} for the basic concepts on group theory.

Let $v,t$ be positive integers with $v \equiv 0 \pmod{2t}$, let $G$ be an abstract group of order $v-2$ acting by right translation on the
set $\ol{G}=G\ \cup \ \{\infty, \infty'\}$ where $\infty$ and $\infty'$ are distinct elements not belonging to $G$ and fixed by the action. Given a cycle $C$ in $K_{\ol{G}}$ and an element $g \in G$, we denote by
$C+g$ the \emph{translate of $C$ by $g$}, namely, the cycle obtained from $C$ by replacing each vertex $x\in V(C)\setminus\{\infty,\infty'\}$ with $x+g$. 
Also, the \emph{list of differences} of $C$ is the 
multiset $\Delta C$ of all possible differences $x - y$ (or quotients $xy^{-1}$) between
two adjacent vertices $x$ and $y$ of $C$ with $x,y \not\in \{\infty, \infty'\}$. 
In other words, if $C=(x_1, x_2, \ldots, x_{\ell})$ and $G$ is an additive group, then
\[\Delta C = \{\pm(x_i - x_{i+1}) \;|\; x_i, x_{i+1}\not\in\{\infty, \infty'\}, i=1,2, \ldots, \ell\},\]
where $x_{\ell+1}=x_1$. Given a family $\mathcal{S}$ of cycles we denote the list of differences of $\mathcal{S}$ by 
$\Delta \mathcal{S} = \cup_{C\in \mathcal{S}} \Delta C$.

Given a $CS(v,2t)$ $\mathcal{C}$, we say that $\mathcal{C}$ is \emph{$2$-pyramidal} over $G$ if $K_v=K_{\ol{G}}$ 
and the translation by any element of $G$ preserves the cycle--set $\mathcal{C}$, that is, 
for any cycle $C \in \mathcal{C}$ and any $g \in G$, we have that $C+g\in \mathcal{C}$.
Note that the group of right translations induced by $G$ is an automorphism group of $\mathcal{C}$ fixing two
vertices, $\infty$ and $\infty'$, and acting sharply transitively on the others. 
Cycle systems of this type, and more generally,  $k$-pyramidal cycle systems over an abstract group $G$ have been studied, for example, in \cite{BoMaRi09, BuTr12, BBRT14, BuTr15}. A $1$--pyramidal cycle system is usually called \emph{$1$--rotational} and recent existence results can be found in \cite{BoBuRiTr12,BuRi08, BuRiTr13, Ol05, Tr10, Tr13}.

Given a cycle $C$ of $K_{\ol{G}}$, we denote by $Orb(C)=\{C+g \;|\;g\in G\}$ the {$G$-orbit} of $C$, namely, the set
of all distinct translates of $C$. Of course, $Orb(C)$ has size at most $|G|=v-2$.  

We say that $C$ is a \emph{short-orbit-cycle} if  $C$ is fixed by a non-zero element of $G$, that is, $C+g=C$ for some 
$g\in G\setminus\{0\}$. Otherwise, $C$ is called a \emph{long-orbit-cycle}. It is not difficult to check that the 
$G$-orbit of a short-orbit-cycle has size at most $(v-2)/2$, whereas the $G$-orbit of a long-orbit-cycle has size $v-2$. Finally, we recall that an element of order $2$ of $G$ is sometimes called an \emph{involution}.

\begin{defn}\label{starter} Let $v \equiv 0 \pmod{2t}$ and let $G$ be a group of order $v-2$.
A set $\mathcal{S}$ of $2t$-cycles of $K_{\ol{G}}$ containing $\ell$ long-orbit-cycles and $s$ short-orbit-cycles is called
a \emph{$(K_{\ol{G}}, 2t)$-starter system} if the following conditions are satisfied:
\begin{enumerate}
  \item there is exactly one cycle $C\in\mathcal{S}$ (resp., $C'\in\mathcal{S}$) containing $\infty$ 
  (resp., $\infty'$)  and $C$ is a short-orbit-cycle;
  \item $\Delta \mathcal{S}\supseteq G\setminus\{0,\lambda\}$ where $\lambda$ is an involution of $G$;  
  \item $s+2\ell \leq v/2t$;
\end{enumerate} 
\end{defn}

The following Lemma shows that any $(K_{\ol{G}}, 2t)$-starter system yields a $2$-pyramidal $CS(v,2t)$ over $G$.

\begin{lem}\label{2pyramidal} Let $v \equiv 0 \pmod{2t}$ and let $G$ be a group of order $v-2$.
If there exists a $(K_{\ol{G}}, 2t)$-starter system $\mathcal{S}$, then
the set $\cup_{C\in \mathcal{S}} Orb(C)$ is
a $2$-pyramidal $CS(v,2t)$ over $G$.
\end{lem}
\begin{proof} Let  $\mathcal{S}= \{C_1, C_2, \ldots, C_{s+\ell}\}$ be a {$(K_{\ol{G}}, 2t)$-starter system} and assume, 
without loss of generality, that $C_1, \ldots, C_s$ are the $s$ short-orbit-cycles of $\mathcal{S}$. 

We first show that no short-orbit-cycle of $\mathcal{S}$ contains the edge $\{\infty, \infty'\}$. By contradiction,
assume that there is a short-orbit-cycle of $\mathcal{S}$, say $C_1$, where  $\infty$ and $\infty'$ are adjacent, hence, $C_1=(\infty, \infty', x,\ldots)$. Since $C_1$ has a short $G$-orbit, there is $g\in G\setminus\{0\}$ such that
$C_1+g=(\infty, \infty', x+g,\ldots)=C$. It follows that $x+g=x$, that is, $g=0$ which is a contradiction.

We now use properties \ref{starter}.(1)-(2) to show that $\mathcal{C}$ contains all the edges of the form
$\{\infty, x\}, \{\infty', x\},$ and $\{x,g+x\}$ for any $x \in G$ and 
for any $g\in G\setminus\{0, \lambda\}$. Assume that $\infty$ lies in $C_i$, hence $C_i=(\infty, a,\ldots)$. Then, $\{\infty,x\}$ 
is an edge of $C_i + (-a+x)$ for any $x \in G$. A similar reasoning shows that the edge $\{\infty',x\}$ lies in some cycle of $\mathcal{C}$ for any $x \in G$. 
Now, let $x\in G$, $g \in G\setminus\{0,\lambda\}$ and let $C_j$ be the cycle of $\mathcal{S}$ such that
$g \in \Delta C_j$; therefore, $C_j = (y, g+y, \ldots)$ for some $y\in G$. It follows that
$C_j + (-y+x)$ contains the edge $\{x, g+x\}$.

We have therefore proven that the cycle-set $\mathcal{C}=\cup_{i=1}^{\ell+s} Orb(C_i)$ covers (at least)
all the edges of $K_{\ol{G}}-I$ where $I=\{\{x,\lambda+x\}\;|\;x\in G\}\ \cup \ \{\{\infty, \infty'\}\}$ is a $1$-factor of $K_{\ol{G}}$.
Note that the size of $K_{\ol{G}}-I$ is $v(v-2)/2$. Therefore, to prove that $\mathcal{C}$ is a cycle system of
$K_{\ol{G}}-I$, we only need to show that $\mathcal{C}$ covers at most $v(v-2)/2$ edges. Denote by $E$ the number of edges of $K_{\ol{G}}$ contained in some cycle of $\mathcal{C}$, then 
\[
E \leq \sum_{i=1}^{\ell+s} |E(C_i)||Orb(C_i)| = 2t \cdot \sum_{i=1}^{\ell+s} |Orb(C_i)|.  
\]
Now, recall that either $|Orb(C_i)|\leq (v-2)/2$ or $|Orb(C_i)| =  v-2$ according to whether $C_i$ is a short-orbit or a long-orbit-cycle, therefore,
\[
  E \leq 2t\cdot\sum_{i=1}^{\ell+s} |Orb(C_i)| \leq 
  2t\left(s\frac{v-2}{2} + \ell (v-2)\right) = 
  2t(s+2\ell)\left(\frac{v-2}{2}\right).
\]
Finally, by property \ref{starter}.(3), we have that $s+2\ell\leq v/2t$, hence $E\leq v(v-2)/2$ and this completes the proof.
\end{proof}

We now deal with $2t$-cycle systems of the complete bipartite graph with parts of the same size.
The techniques described below are used in Section 5 to construct intersecting 
$2t$-cycle systems of some complete bipartite graphs. We point out that although 
the existence of a $CS(K_{r,s}, k)$ is completely solved in \cite{So81}, the related constructions
do not yield intersecting systems.

Let $w \geq t$ be positive integers with $w^2 \equiv 0 \pmod{2t}$ so that the obvious necessary conditions for the existence of a $CS(K_{w,w},2t)$ are satisfied. Also,  denote by $\Z_n$ the cyclic group of order $n$. 
From now on, we let $K_{w,w}$ be the complete bipartite graph with vertex-set $\Z_{w}\times \Z_2$ and parts $X=\Z_{w}\times \{0\}$
and $Y=\Z_{w}\times \{1\}$. 

Given a cycle $C$ of $K_{w,w}$, the \emph{list of oriented differences of $C$}
is the multiset $\oDelta C$ containing the differences $g-h$ between the adjacent vertices $(g,1)$ and $(h,0)$ of $C$, namely,
\[
  \oDelta C= \{g-h \;|\; \{(g,1), (h,0)\}\in E(C).\}
\]

Given a set $\mathcal{S}$ of cycles of $K_{w,w}$, we denote by 
$\oDelta \mathcal{S} = \cup_{C\in\mathcal{S}} \oDelta C$ the list of oriented differences of $\mathcal{S}$.
Also, for any $g \in \Z_{w}$ we denote by $C+g$ the cycle obtained by replacing each vertex of $C$, say $(z,i)$,
with $(z+g,i)$. Finally, let $Orb(C)=\{C+g\;|\; g\in\Z_{w}\}$ be the orbit of $C$ over $\Z_{w}$.

Given a $CS(K_{w,w}, 2t)$ $\mathcal{C}$, we say that $\mathcal{C}$ is \emph{semiregular} over $\Z_{w}$ if 
the translation by any element of $\Z_{w}$ preserves the cycle--set $\mathcal{C}$, that is, 
for any cycle $C \in \mathcal{C}$ and any $g \in G$, we have that $C+g\in \mathcal{C}$.
As before, the group of right translations induced by $\Z_{w}$ on  $\Z_{w}\times \Z_2$ 
is an automorphism group of $\mathcal{C}$ acting semiregularly  on the vertex-set.

We now define a slightly different starter system that allows us to construct semiregular $2t$-cycle systems
of $K_{w,w}$.

\begin{defn}\label{starter2} Let $w \geq t$ be an integer such that $w^2 \equiv 0 \pmod{2t}$.
A set $\mathcal{S}$ of $2t$-cycles $K_{w,w}$  is called
a \emph{$(K_{w,w}, 2t)$-starter system} if $\oDelta \mathcal{S}$ covers the elements in 
$\Z_{w}$ exactly once, i.e., $\oDelta \mathcal{S} = \Z_{w}$.
\end{defn}

The following Lemma shows that any $(K_{w,w},2t)$-starter system yields a semiregular 
$CS(K_{w,w},2t)$ over $\Z_{w}$.

\begin{lem}\label{semiregular} Given a $(K_{w,w},2t)$-starter system $\mathcal{S}$, 
the set $\mathcal{C}=\cup_{C\in \mathcal{S}} Orb(C)$ is
a semiregular $CS(K_{w,w},2t)$ over $\Z_{w}$.
\end{lem}
\begin{proof} We first show that an edge of $K_{w,w}$ lies 
in at most one cycle of $\mathcal{C}$. By contradiction, assume that the edge 
$\{(z,1),(z',0)\}$ of $K_{w,w}$ is contained in two distinct cycles of $\mathcal{C}$, say $C_i+g_i$ 
with $C_i \in \mathcal{S}$, $g_i \in \Z_w$, and $i=1,2$. If $C_1 \neq C_2$, 
then $z-z' \in \oDelta(C_i+g_i)=\oDelta C_i$ for $i=1,2$ hence, $z-z'$ has multiplicity at least $2$ in
$\oDelta \mathcal{S}$. If $C_1 = C_2$, then $g_1\neq g_2$ and the edges 
$\{(z-g_i,1),(z'-g_i,0)\}$ for $i=1,2$ are distinct edges of $C_1$ which give rise to the same difference $z-z'$; therefore,  
$z-z'$ has multiplicity at least $2$ in $\oDelta \mathcal{S}$. In both cases, $\oDelta \mathcal{S}$ has repeated elements
contradicting our assumption.

It is left to show that any edge of $\{(z,1),(z',0)\}$ of $K_{w,w}$ is contained in a cycle of $\mathcal{C}$.
Since $z-z' \in \oDelta \mathcal{S}$, there is a cycle $C$ of $\mathcal{S}$ such that $z-z' \in \oDelta C$, hence
$C=((g,1),(g',0), \ldots,)$ where $g-g' = z-z'$.   
Since $C + (-g'+z') = ((z,1), (z',0), \ldots)$, we have that 
$\{(z,1),(z',0)\}$ is contained in a cycle of $\mathcal{C}$ and this completes the proof.
\end{proof}

We end this section with a lemma that we will use in Sections 4 and 5 to check the intersecting property 
for a $2$-pyramidal or a semiregular cycle system.
First, given two sets $V_1$ and $V_2$ of $\Z_{n}$, we denote by $V_1-V_2$ the set defined as follows: 
$V_1-V_2 = \{v_1-v_2 \;|\; (v_1,v_2)\in V_1\times V_2\}$.

\begin{lem} \label{X1-X2} Given two subsets $V_1$ and $V_2$  of $\Z_w$, we have that
 $(V_1+g_1) \ \cap \ (V_2+g_2)\neq \emptyset$ for any $g_1, g_2 \in\Z_v$ if and only if 
$V_1-V_2 \supseteq \Z_w$.
\end{lem}
\begin{proof} We first note that $(V_1+g_1) \ \cap \ (V_2+g_2)$ is non-empty if and only if 
$V_1\ \cap \ (V_2+g_2-g_1)$ is non-empty. Therefore, it is enough to show that 
$V_1\ \cap \ (V_2+g)\neq \emptyset$ if and only if  $g\in V_1-V_2$.

If $g \in V_1-V_2$, then there is a pair $(x_1, x_2) \in V_1 \times V_2$ such that $x_1 - x_2 = g$; hence,
$x_1 = x_2+g \in V_1\ \cap \ (V_2+g)\neq \emptyset$. Conversely, if 
$V_1\ \cap \ (V_2+g)\neq \emptyset$, then there is a pair $(x_1, x_2) \in V_1 \times V_2$ such that 
$x_1 = x_2+g$; hence, $g=x_1-x_2 \in V_1-V_2$.
\end{proof}


\section{Unparalleled $CS(v,4)$}\label{intersecting1}
In this section we show the existence of unparalleled $4$-cycle systems. 
We slightly modify a construction of \cite{BuRi08} which allows the authors to characterize 
$1$-rotational $2$-factorizations of the complete graph under the dihedral group. As pointed out 
in \cite{BuTr15},
any $1$-rotational $2$-factorization of $K_v$ yields a $2$-pyramidal $2$-factorization of $K_{v+1}-I$.
Our slightly different construction will lead to an unparalleled $CS(v,4)$ which is $2$-pyramidal under the dihedral group, for any $v \equiv 0 \pmod{4}$. 

We denote by $D_{4s-2}$ the dihedral group of order $4s-2$, i.e., the group with defining relations
\[
  D_{4s-2} = \langle x,\lambda\;|\; x^{2s-1} = \lambda^2 = 1; \lambda x = x^{-1}\lambda\rangle.
\]
Note that $D_{4s-2}$ is a multiplicative group, hence we denote by $1$ its unit. 
Also, we denote by $\langle x \rangle=\{x^0=1, x, x^2, \ldots, x^{2s-2}\}$ the subgroup of $D_{4s-2}$ generated by $x$. We collect here some known properties of $D_{4s-2}$ (see, for example, \cite{Sc87}):
\begin{enumerate}
\item[D1.] $D_{4s-2} = \{1, x^j\lambda\}\cdot \langle x \rangle$ for any integer $j$;
\item[D2.] $\langle x \rangle\setminus\{1\} = \{x^{\pm 2i} \;|\; i=1,2\ldots,s-1\}$;
\item[D3.] $(x^{j}\lambda)^2=1$ and $x^{j}\lambda = \lambda x^{-j}$  for any integer $j$.
\end{enumerate}

\begin{thm}\label{four}
There exists an unparalleled $CS(4s,4)$ for any $s>1$.
\end{thm}
\begin{proof} Let $\mathcal{S}=\{C_0, C_1, \ldots, C_{s-1}\}$ be a set of $s$
$4$-cycles of $K_{\ol{D_{4s-2}}}$ defined as follows:
\begin{align*}
  & C_0 = (\infty, 1, \infty', x\lambda) \;\;\;\text{and}\;\;\; 
  C_i = (x^{i}, x^{-i}, \lambda x^{i}, \lambda x^{-i}) \;\;\;\text{for}\;\;\; i=1,\ldots, s-1.
\end{align*}
Note that each of these cycles has a short $G$-orbit: in fact, by Property D3, we have that
\begin{equation}\label{4cycles}
\begin{aligned}
  & C_0\cdot x\lambda = (\infty, x\lambda, \infty', (x\lambda)^2=1) = C_0, \\ 
  & C_i\cdot \lambda = (x^{i}\lambda, x^{-i}\lambda, \lambda x^{i}\lambda, \lambda x^{-i}\lambda) = 
(\lambda x^{-i}, \lambda x^i, x^{-i}, x^i) = C_i, 
\end{aligned}
\end{equation}
for any $i=1,\ldots, s-1$.
Also, $\Delta C_i \supseteq 
\{x^{\pm 2i}, \lambda x^{2i}, \lambda x^{-2i}\} = \{1,\lambda\} \cdot x^{\pm2i}$
for any $i=1,\ldots, s-1$. Therefore, by Properties D1 and D2, it follows that
\[
\Delta \mathcal{S} = \bigcup_{i=1}^{s-1} \Delta C_i \supseteq 
\{1,\lambda\}\cdot(\langle x\rangle\setminus\{1\}) = D_{2n}\setminus\{1, \lambda\}.\]
Thus, we have proven that $\mathcal{S}$ is a $(D_{4s-2},4)$-starter system and, 
by Lemma \ref{2pyramidal}, we have that $\mathcal{C}=\cup_{C\in \mathcal{S}} Orb(C)$ is
a $2$-pyramidal $CS(4s,4)$ over $D_{4s-2}$.

It is left to show that $\mathcal{C}$ is free from parallel classes. 
By contradiction, assume that
$\mathcal{C}$ has a parallel class $\mathcal{P}$. We consider $\mathcal{P}$ as a $2$-factor of 
$K_{\ol{D_{4s-2}}}$ hence, its vertex-set is  $V(\mathcal{P})=D_{4s-2}\ \cup \ \{\infty, \infty'\}$.
We note that by \eqref{4cycles} it follows that $C_0\cdot x\lambda a=C_0\cdot a$
and $C_i\cdot \lambda a= C_i \cdot a$ for any $i\geq 1$ and $a\in\langle x\rangle$. 
Hence, by Property D1, we have that any translate of $C_i$ is of the form $C_i\cdot a$ 
for some $a \in \langle x \rangle$. So,
$Orb(C_i) = \{C_i\cdot a \;|\; a\in\langle x \rangle\}$ for any $i\geq 0$.

Since $\mathcal{P}$ contains 
translates of some cycles of $\mathcal{S}$, we denote by 
$\{a_{i1},a_{i2}, \ldots, a_{i\ell_i}\}$ 
the (possibly empty) set of all elements $a_{ij}\in\langle x \rangle$ such that $C_i \cdot a_{ij}\in \mathcal{P}$; hence
$\mathcal{P} = \cup_{i=0}^{s-1} \cup _{j=1}^{\ell_i} C_i \cdot a_{ij}$. 
Note that all cycles in $Orb(C_0)$ contains $\infty$ and 
$\infty'$, hence $\mathcal{P}$ contains only one translate of $C_0$, i.e., $\ell_0=1$. 
Therefore,
\begin{align*}
  & V(\mathcal{P}) = \bigcup_{i=0}^{s-1} \bigcup_{j=1}^{\ell_i} 
  V(C_i \cdot a_{ij}) = \{\infty, \infty', a_{01},x\lambda a_{01}\}\ \cup\ U \ \cup\ W, \\
  & \text{where} \;\;U= \bigcup_{i=1}^{s-1} \bigcup_{j=1}^{\ell_i} \{x^{i}, x^{-i}\}\cdot a_{ij} \;\;\;\text{and}\;\;\;
  W = \bigcup_{i=1}^{s-1} \bigcup_{j=1}^{\ell_i} \{\lambda x^{i},\lambda x^{-i}\}\cdot a_{ij}.
\end{align*}
Note that $\{a_{01}\} \ \cup \ U  \subseteq \langle x \rangle$,  
$x\lambda a_{01} = \lambda x^{-1}a_{01}\in \lambda\cdot\langle x \rangle$ and 
$W=\lambda\cdot U \subseteq \lambda\cdot\langle x \rangle$.
Since $V(\mathcal{P})\setminus\{\infty, \infty'\} = {D_{4s-2}} = \{1,\lambda\}\cdot \langle x \rangle$, 
it follows that $\langle x \rangle = \{a_{01}\} \ \cup \ U$, that is, $U =\langle x \rangle\setminus\{a_{01}\}$. Also, 
\[\lambda\cdot \langle x \rangle = 
\{x\lambda a_{01}\} \ \cup \ W =
\{x\lambda a_{01}\} \ \cup \ \lambda\cdot U =
\{x\lambda a_{01}\} \ \cup \ (\lambda\cdot \langle x \rangle\setminus\{\lambda a_{01}\}).\]
Hence, $x\lambda a_{01} = \lambda a_{01}$, that is, $x=1$ which is a contradiction since $x$ is an element
of order $2s-1>1$. 
\end{proof}

\section{Intersecting CS$(v,2t)$ for $v=4t$ and $6t$.}
\label{intersecting2}
We say that a $2t$-cycle system $\mathcal{C}$ of $K_v-I$ 
is \emph{intersecting} if any two cycles of ${\cal C}$ intersect in at least one vertex.
In this section we prove the existence of an intersecting CS$(v,2t)$ when either $v=4t$, or $v=6t$ and $t$ is odd. In both cases, we construct an intersecting cycle system $\mathcal{C}$ of $K_v-I$ which is 
$2$-{\it pyramidal} under the cyclic group $\Z_{v-2}$. Each cycle is obtained from paths with a suitable
list of differences. We point out that we will use the same notation as the one for cycles to denote a path.
To avoid misunderstandings, we will always specify whether we are
dealing with a path or a cycle. 
Also, given a path on $t$ vertices (\emph{$t$-path}) $P=(a_1, a_2, \ldots,a_t)$ with  $a_i\in\Z_{v-2}$, its list of differences is the multiset $\Delta P = \{\pm(a_i-a_{i+1})\;|\; i=1, \ldots,t-1\}$. 

We point out that here and in Section 5 we will denote by $[a, b]$ the \emph{interval} of integers 
$\{a, a+1, a+2, \ldots, b\}$ for $a\leq b$,  Of course, if $a < b$, then $[b, a]$ will be the empty set.

We first deal with the existence of an intersecting $CS(4t,2t)$ and prove the
following preliminary result.

\begin{lem}\label{AB}
For $t \geq 3$, there exist a $t$-path $A=(a_1, a_2, \ldots, a_{t})$ and
a $(t-2)$-path $B=(b_1, b_2, \ldots, b_{t-2})$ with vertices in $\Z$
such that $a_t-a_1=\pm 1$, $b_1=2t-2$, and $\Delta A \ \cup \ \Delta B = \pm[2, 2t-3]$.
Moreover,
\begin{align*}
  \text{if $t=4s-1$,\;} &   
    V(A) = [0,4s-1]\setminus\{3s-1\}, V(B) = [1,2s-2]\cup[6s-2, 8s-4], \\
  \text{if $t=4s$,\;} & 
    V(A) = [0,4s]\setminus\{3s\}, V(B) = [1,2s-1]\cup[6s, 8s-2], \\    
  \text{if $t=4s+1$,\;} &   
    V(A) = [0,4s+2]\setminus\{3s, 3s+1\}, \\
  & V(B) = [1,2s-1]\ \cup\ \{4s-2\}\ \cup\ [6s+2, 8s], \\ 
  \text{if $t=4s+2$,\;} &   
    V(A) = [0,4s+3]\setminus\{3s+1, 3s+2\},\\
  & V(B) = [1,2s-1]\ \cup\ \{4s+3\}\ \cup\ [6s+3, 8s+2]. 
\end{align*}
\end{lem}
\begin{proof} Let $t=4s+i\geq 3$ for $i\in\{-1,0,1,2\}$. For each value of $i$, we provide below
the required paths $A=A_i$ and $B=B_i$: 
\begin{align*}
A_{-1} = (&2s-1, 0,4s-1,1,4s-2, \ldots, i, 4s-1-i, \ldots, s-1, 3s, \\
& s,3s-2,s+1, 3s-3, \ldots, i, 4s-i-2, \ldots, 2s-2, 2s), \\
B_{-1} = (& 8s-4, 1, 8s-5, 2, \ldots, 8s-3-i, i, \ldots, 6s-1, 2s-2, 6s-2),\\
A_0 = (&2s, 0,4s,1,4s-1, \ldots, i, 4s-i, \ldots, s-1, 3s+1, \\
& s,3s-1,s+1, 3s-2, \ldots, i, 4s-i-1, \ldots, 2s-2, 2s+1, 2s-1), \\
B_0 = (& 8s-2, 1, 8s-3, 2, \ldots, 8s-1-i, i, \ldots, 6s, 2s-1),\\
A_1 = (&2s, 0,4s+2, 1,4s+1, \ldots, i, 4s+2-i, \ldots, s, 3s+2, \\
&s+1,3s-1,s+2, 3s-2, \ldots, i, 4s-i, \ldots, 2s-1, 2s+1), \\
B_1 = (&8s, 1, 8s-1, 2,\ldots, 8s+1-i,i, \ldots, 6s+2, 2s-1,4s-2),\\
A_2 = (&2s+1, 0,4s+3, 1,4s+2, \ldots, i, 4s+3-i, \ldots, s, 3s+3, \\
&s+1,3s,s+2, 3s-1, \ldots, i, 4s-i+1, \ldots, 2s-1, 2s+2, 2s), \\
B_2 = (&8s+2, 1, 8s+1, 2,\ldots, 8s+3-i,i, \ldots, 6s+4, 2s-1, 6s+3,4s+3).
\end{align*}
It is not difficult to check that\\

\noindent
\begin{tabular}{ll}
   $\Delta A_{-1}$ = $\pm[2,4s-1]$, & $\Delta B_{-1}$ = $\pm[4s,8s-5]$;\\  
   $\Delta A_0\;\;   $ = $\pm[2,4s]$,   & $\Delta B_0$\;\;    = $\pm[4s+1,8s-3]$;\\
   $\Delta A_1\;\;   $ = $\pm([2,4s+2]\setminus\{2s-1\})$, & $\Delta B_1$\;\;  = 
   $\pm([4s+3,8s-1]\cup\{2s-1\})$;\\
   $\Delta A_2\;\;   $ = $\pm([2,4s+3]\setminus\{2s\})$,   & $\Delta B_2$\;\;  = 
   $\pm([4s+4,8s+1]\cup\{2s\})$.
\end{tabular}\\

Therefore, $\Delta A \ \cup \ \Delta B = \pm[2, 2t-3]$. We leave the reader to check that $A$ and $B$ 
satisfy all the remaining properties.
\end{proof}

We are now able to prove the following result.

\begin{thm}\label{int1}
   There exists an intersecting $CS(4t,2t)$ for any $t\geq 3$.
\end{thm}
\begin{proof}  
  Consider a $t$-path $A=(a_1, a_2, \ldots, a_{t})$, with $a_t-a_1=\pm1$, and a  
  $(t-2)$-path $B=(b_1, b_2, \ldots,$ $b_{t-2})$, with $b_1 = 2t-2$, 
  as in Lemma \ref{AB}. Now, we construct two 
  $(2t)$-cycles $C_1$ and $C_2$ with vertices in $\Z_{4t-2} \ \cup \ \{\infty, \infty'\}$ as follows:
  \begin{align*}
    C_1 = (&a_1, a_2, \ldots, a_{t}, 
    a_1+2t-1, a_2+2t-1, \ldots, a_{t}+2t-1),\\
    C_2 = (&\infty, b_0, b_1, \ldots, b_{t-2}, \infty', 
     b_{t-2} + 2t-1, \ldots, b_1+2t-1, b_0+2t-1).
  \end{align*}
  where $b_0=2t-1$. By construction $C_i+2t-1=C_i$ for $i=1,2$. Also,
  \begin{align*}
  \Delta C_1  \supseteq &\; \Delta A 
  \ \cup \ \{\pm(a_t-a_1)+2t-1\} = \Delta A \ \cup \ \{\pm(2t-2)\} \\
  \Delta C_2  \supseteq &\; \Delta B \ \cup \ \{\pm (b_1-b_0)\} = 
  \Delta B \ \cup  \ \{\pm 1\}.
  \end{align*}
  By Lemma \ref{AB} we have that $\Delta A \ \cup \ \Delta B = \pm[2, 2t-3]$. Therefore, 
  $\Delta C_1 \ \cup \ \Delta C_2$ contains  
  $\pm[1,2t-2] = \Z_{4t-2}\setminus\{0,2t-1\}$. Note that $2t-1$ is the involution of 
  $\Z_{4t-2}$ which, by construction, fixes both $C_1$ and $C_2$. It follows that $\{C_1, C_2\}$
  is a $(K_{\ol{\Z_{4t-2}}}, 2t)$-starter system hence, by Lemma \ref{2pyramidal}, 
  ${\cal C} = Orb(C_1) \ \cup \ Orb(C_2)$ is a $2$-pyramidal $CS(4t,2t)$. 
  
  It is left to show that ${\cal C}$ is intersecting. We will prove the assertion only 
  for $t=4s$  and leave the easy check of the other cases to the reader. Setting
  $V_1 = V(C_1)$ and $V_2=V(C_2)\setminus\{\infty, \infty'\}$, we have by Lemma \ref{AB} that
  \begin{align*}
    V_1 =&\;  
    ([0,4s] \ \cup \ [8s-1,12s-1]) \setminus \{3s,11s-1\} \\ 
    V_2 =&\;   [0,2s-1] \ \cup \ [6s,10s-2] \ \cup \ [14s-1, 16s-3].
  \end{align*}
  It is not difficult to check that $\Z_{4t+2} \subseteq V_1 - V_i$ for any $i\in \{1,2\}$. Now, by taking into account
  Lemma \ref{X1-X2} and considering that $\infty, \infty' \in V(C_2)+x$ for any $x\in \Z_{4t+2}$, we conclude that any two cycles in ${\cal C}$ 
  share some vertices, that is, ${\cal C}$ is intersecting.   
\end{proof}

Finally, we construct an intersecting CS$(6t,2t)$ for any odd $t\geq 3$.

\begin{prop} \label{int2}
   There exists an intersecting $CS(6t,2t)$ for any odd $t\geq 3$.
\end{prop}
\begin{proof} We first deal with the case $t=3$ and consider the two $6$-cycles
 $C_1=(\infty, 0,2,\infty', 10, 8)$ and $C_2=(0,3,8,7,14,4)$ with vertices in $\Z_{16}$.
 Note that $\Delta C_1 \supseteq \{\pm 2\}$ and 
 $\Delta C_2 = \Z_{16}\setminus\{0,\pm2,8\}$; thus, 
 $\Delta C_1 \ \cup\ \Delta C_2 \supseteq \Z_{16}\setminus\{0,8\}$. Also, $C_1 + 8 = C_1$.   Therefore, 
 $\{C_1, C_2\}$ is a $(K_{\ol{\Z_{16}}}, 2t)$-starter system 
  and, by Lemma  \ref{2pyramidal}, we have that 
  ${\cal C} = Orb(C_1) \ \cup \ Orb(C_2)$ is a $2$-pyramidal $CS(18,6)$.
  We now show that ${\cal C}$ is intersecting. First note that all translates of $C_1$ contain 
  $\infty$ and $\infty'$. Now, set $V_1 = V(C_1)\setminus\{\infty, \infty'\}$ and $V_2 = V(C_2)$.
  It is not difficult to check that $V_i-V_2\supseteq \Z_{16}$ for $i=1,2$. Therefore, 
  in view of Lemma \ref{X1-X2}, we have that any two cycles of ${\cal C}$ intersect in at least one    vertex. 

We proceed similarly when $t>3$.
Let $t=2s+1$ with $s>1$, and let $A=(a_1, a_2, \ldots, a_{t-1})$ be the $(t-1)$-path defined by
$A=(0,1,-1,2,-2,\ldots, s-1,-(s-1), 3s+5)$. Now, we construct two 
  $(2t)$-cycles $C_1$ and $C_2$ with vertices in $\Z_{6t-2} \ \cup \ \{\infty, \infty'\}$ as follows:
  \begin{align*}
    C_1 = (& \infty, a_1, a_2, \ldots, a_{t-1}, \infty', 
           a_{t-1}+6s+2, \ldots, a_2+6s+2, a_{1}+6s+2),\\
    C_2 = (& 0, 10s+5, 1, 10s+4,\ldots, i, 10s+5-i, \ldots, s+1, 9s+4,\\
           &   s+2, 9s+1, s+3, 9s, \ldots, j, 10s+3-j, \ldots, 2s-1, 8s+4, 2s, 8s+1).
  \end{align*}
  It is not difficult to check that $\Delta C_1 \supseteq \Delta A = 
                   \pm [1,2s-2] \ \cup \ \{\pm (4s+4)\}$ and 
  $\Delta C_2 = \pm([2s-1, 4s+3] \ \cup \ [4s+5,6s+1])$;
  hence, $\Delta C_1 \ \cup \ \Delta C_2 \supseteq \pm[1,6s+1]= \Z_{6t-2}\setminus\{0, 3t-1 = 6s+2\}$.
  Note that $6s+2$ is the involution of $\Z_{6t-2}$ which fixes, by construction, $C_1$. 
  It is then straightforward to check that $\{C_1, C_2\}$ is a $(K_{\ol{\Z_{6t-2}}}, 2t)$-starter system 
  and, by Lemma   
  \ref{2pyramidal}, we have that 
  ${\cal C} = Orb(C_1) \ \cup \ Orb(C_2)$ is a $2$-pyramidal $CS(6t,2t)$. 
  
  It is left to show that ${\cal C}$ is intersecting.  For our convenience we set 
  $V_1 = V(C_1)\setminus\{\infty, \infty'\}$ and $V_2 = V(C_2)$. 
  Since $\infty$ and $\infty'$ lie in any  translate of $C_1$, in view of Lemma \ref{X1-X2}
  we only need to show that   $\Z_{6t-2} \subseteq V_2 - V_i$ for $i=1,2$. 
  We have that
  \begin{align*}
    &V_1 = (W \ \cup \ \{3s+5\}) +\{0,6s+2\}, \; \text{and},\\
    &V_2 = [0,2s] \ \cup \ \{8s+1\} \ \cup \ [8s+4,10s+5]\setminus\{9s+2,9s+3\}, 
  \end{align*}
   where $W= [-(s-1), s-1]$. It is not difficult to check that $\Z_{6t-2} \subseteq V_2-V_2$. 
   Now, note that $-V_1 = (W \ \cup \ \{-3s-5\}) +\{0,6s+2\}$,
   therefore,
   \[ \Z_{6t-2}\subseteq (V_2 + W) + \{0,6s+2\} \subseteq V_2 - V_1.
   \]
   Therefore, ${\cal C}$ is intersecting.
\end{proof}

\section{Intersecting $CS(K_{w,w}, 2t)$ for $w=2t$ and $4t$}
We denote by $K_{w,w}$ the complete bipartite graph with vertex-set $\Z_{w}\times \Z_2$ and
parts $X= \Z_{w}\times \{0\}$ and $Y=\Z_{w}\times \{1\}$. Also, we say that a
$2t$-cycle system $\mathcal{C}$ of $K_{w,w}$ 
is \emph{intersecting} if any two cycles of ${\cal C}$ intersect in at least one vertex of $X$, that is,
\begin{equation}\label{good}
V(C) \ \cap\ V(C') \ \cap\ X \neq \emptyset \quad\quad \text{for any}\quad  C,C' \in \mathcal{C}.
\end{equation}

In this section we prove the existence of an intersecting $CS(K_{w,w}, 2t)$ when either $w=2t$ and $t\geq4$ is even, or $(w,t)=(2t,3),(2t,5)$, or $w=4t$ and $t\geq 7$ is odd.

For the sake of brevity, we denote by $x_i$ the vertex $(x,i)\in \Z_{w}\times \Z_2$.

\begin{thm}\label{strong1}
   There exists an intersecting $CS(K_{2t,2t}, 2t)$  for any even $t\geq4$.
\end{thm}
\begin{proof} 
   Let $t=2s$ with $s \geq 2$, 
   and denote by $C$ the following $2t$-cycle:
   \begin{align*}
     C=  (&(4s-1)_1, 1_0, (4s-2)_1, 2_0, \ldots, (4s-i)_1, i_0, \ldots, (3s)_1, s_0 \\
          &(s-1)_1, (3s)_0, (s-2)_1, (3s+1)_0, \ldots, (s-j)_1, (3s-1 +j)_0, \ldots, \\
          & 1_1, (4s-2)_0, 0_1, (4s-1)_0).
   \end{align*}   
   We can easily check that $\oDelta C = \Z_{2t}$; therefore, $C$ is 
   a $(K_{2t,2t},2t)$-starter system and, by Lemma \ref{semiregular}, $\mathcal{C}= Orb(C)$ is a semiregular 
   $CS(K_{2t,2t}, 2t)$ over $\Z_{2t}$.
  
  It is left to show that $\mathcal{C}$ satisfies Property \eqref{good}.
  If we set $V = V(C) \ \cap \ X$, this is equivalent to showing that $(V+g) \ \cap \ (V+h) \neq \emptyset$ for
  any $h,g\in {\Z}_{2t}$. 
    Since $V = ([1,s] \ \cup \ [3s,4s-1])\times \{0\}$, it is not difficult to check that 
   ${\Z}_{2t}\subseteq V-V$ and, by Lemma \ref{X1-X2}, we obtain the assertion. 
\end{proof}


\begin{thm}\label{strong3}
   There exists an intersecting $CS(K_{6,6}, 6)$.
\end{thm}
\begin{proof} 
   The following $6$-cycles all together provide a cycle system of $K_{6,6}$:
   \begin{align*}  
     C_0  = (0_0,0_1,1_0,1_1,2_0,2_1),  &\quad    C_3 = (0_0,3_1,3_0,4_1,1_0,5_1),\\
     C_1  = (2_0,0_1,3_0,1_1,4_0,3_1),  &\quad    C_4  = (2_0,4_1,5_0,2_1,3_0,5_1),\\  
     C_2  = (4_0,0_1,5_0,1_1,0_0,4_1),  &\quad    C_5  = (1_0,2_1,4_0,5_1,5_0,3_1).      
   \end{align*}
   It is not difficult to check that any two of these cycles  intersect in a vertex of the form $x_0$.
\end{proof} 

\begin{thm}\label{strong5}
   There exists an intersecting $CS(K_{10,10}, 10)$.
\end{thm}
\begin{proof} 
   The following $10$-cycles all together provide a cycle system of $K_{10,10}$:
   \begin{align*}
     C_0  = (0_0,0_1, 1_0,1_1, 2_0,2_1, 3_0,3_1, 4_0,4_1),&\; C_5=(1_0,5_1, 2_0,6_1, 3_0,7_1, 4_0,    8_1,7_0,9_1),\\
     C_1  = (2_0,0_1, 3_0,1_1, 4_0,2_1, 5_0,3_1, 6_0,4_1),&\; C_6=(3_0,5_1, 4_0,6_1, 5_0,7_1, 6_0,    9_1,9_0,8_1),\\  
     C_2  = (4_0,0_1, 5_0,1_1, 6_0,2_1, 7_0,3_1, 8_0,9_1),&\; C_7=(5_0,5_1, 6_0,6_1, 7_0,7_1, 8_0,    4_1,1_0,8_1),\\     
     C_3  = (6_0,0_1, 7_0,1_1, 8_0,2_1, 9_0,3_1, 0_0,8_1),&\; C_8=(7_0,5_1, 8_0,6_1, 9_0,7_1, 0_0,    9_1,3_0,4_1),\\
     C_4  = (8_0,0_1, 9_0,1_1, 0_0,2_1, 1_0,3_1, 2_0,8_1),&\; C_9=(9_0,5_1, 0_0,6_1, 1_0,7_1, 2_0,    9_1,5_0,4_1).               
   \end{align*}
   It is not difficult to check that any two of these cycles  intersect in a vertex of the form $x_0$.
\end{proof} 

\begin{thm}\label{strong2}
   There exists an intersecting $CS(K_{4t, 4t}, 2t)$ for any odd $t>5$.
\end{thm}
\begin{proof} 
   Let $t=2s+1$ with $s \geq 3$,
   and denote by $C_0$ and $C_1$ the following $2t$-cycles:
   \begin{align*}
     C_0 &= (0_0, 1_1, (2s+1)_0, (2s+3)_1, (4s+2)_0, (4s+5)_1, \\
         &(6s+3)_0, (6s+7)_1, 
         \ldots, (6s+3-i)_0, (6s+7+i)_1, \ldots, (4s+6)_0, (8s+4)_1);
   \end{align*}      
   for $s=3$ set $C_1=(0_0, 26_1, 7_0, 4_1, 12_0, 2_1, 14_0, 13_1, 20_0, 5_1, 21_0, 8_1, 25_0, 14_1)$ 
   and for $s\geq 4$ set
   \begin{align*}
     C_1 &= ((6s+3)_1, 0_0, (8s+1)_1, (2s+1)_0, (2s-1)_1, \\ 
           & (4s+2)_0, (4s+1)_1, (6s+3)_0, (6s-2)_1, \\
           & (8s+3)_0, (6s+6)_1, \ldots, (8s+3-i)_0, (6s+6+i)_1,\ldots, (7s+8)_0, (7s+1)_1, \\
           & (7s+7)_0, (5s)_1, (7s+6)_0, \\
           &(3s+1)_1, (7s+5)_0, \ldots, (3s+j)_1, (7s+6-j)_0, \ldots, (4s-1)_1, (6s+7)_0). 
   \end{align*}  
   We can easily check that 
   \begin{align*}
      \oDelta C_0 =\; &[0,4s-2]\ \cup\ [6s+4,6s+6], \; \text{and} \\
      \oDelta C_1 =\; &[4s-1,6s+3] \ \cup \ [6s+7, 8s+3].   
   \end{align*}   
   Hence, $\oDelta C_0 \ \cup \ \oDelta C_1 = \Z_{4t}$. 
   Therefore, $\{C_1,C_2\}$ is 
   a $(K_{4t,4t},2t)$-starter system and, by Lemma \ref{semiregular}, $\mathcal{C}= Orb(C_1)\ \cup \ Orb(C_2)$ 
   is a semiregular    $CS(K_{4t,4t}, 2t)$ over $\Z_{4t}$.

  It is left to show that $\mathcal{C}$ satisfies property \eqref{good}.
  We recall that $X=\Z_{4t}\times \{0\}$ and $\Z_{4t}\times \{1\}$ are the parts of
  $K_{4t,4t}$. 
  Let $V_i$ be the projection of $V(C_i) \ \cap \ X$ on the first coordinate, for $i=0,1$. 
  It is not difficult to see that $V_i = H \ \cup \ W_i$ where 
  \begin{itemize}
    \item $H=\{0,2s+1,4s+2,6s+3\}$ is the subgroup of ${\Z}_{4t}$ of order $4$,
    \item $W_0 = [4,2s]+4s+2$, and 
    $W_1 = 
    \begin{cases}
     \{12,20,25\} & \text{for $s=3$}, \\
     [4,2s]+6s+3 & \text{for $s \geq4$}. 
    \end{cases}$
  \end{itemize}
  Also, for $i,j\in\{1,2\}$ we have that
  \[
    V_i-V_j \supseteq H \ \cup \ (H-W_j) \ \cup \ (W_i-H) = H \ \cup \ (-W_j+H) \ \cup \ (W_i+H).
  \]
  Since $W_i + H = [4,2s] + H$ and $-W_j + H = [1,2s-3] + H$, we then have that 
  \[
    V_i-V_j \supseteq [0,2s] + H = \Z_{4t},\quad \text{for any $i,j\in \{1,2\}$}.
  \]
  Therefore, by Lemma \ref{X1-X2}, $(V_i + a) \ \cap\ (V_j + b)$ is non-empty for any $i,j\in \{1,2\}$ and for any $a,b \in \Z_{4t}$
  hence, property \eqref{good} is satisfied.
\end{proof}

\section{Proof of Theorem \ref{main}}
This section is devoted to the proof of Theorem \ref{main}.
We start with Construction \ref{constr}, which yields a $CS(u + w, 2t)$,
and will enable the recursive construction of unparalleled cycle systems.

\begin{constr}\label{constr}
  Let $t,u$ and $w$ be positive integers with $u,w\equiv 0 \pmod{2t}$. Also, let
  $\mathcal{U}$ be a $CS(u,2t)$ with vertex-set $U$ and let
  $\mathcal{W}$ be a $CS(w,2t)$ with vertex-set $W$.
  Now, take a partition $\{X_1, X_2, \ldots, X_{\alpha}\}$ of $U$ 
  and a partition $\{Y_1, Y_2, \ldots,Y_{\beta}\}$ 
  of $W$ into subsets whose size is a multiple of $2t$. 
  Finally, choose a $2t$-cycle system $\mathcal{C}_{i,j}$ of $K_{X_i, Y_j}$ 
  for any $i,j$.
  Then, the set $\mathcal{F} = \mathcal{U} \ \cup \ \mathcal{W} \ \cup \ \bigcup_{i,j} \mathcal{C}_{i,j}$
  is a $CS(u+w,2t)$.
\end{constr} 
\begin{figure}[h]
    \centering
    \includegraphics[width=1.0\textwidth]{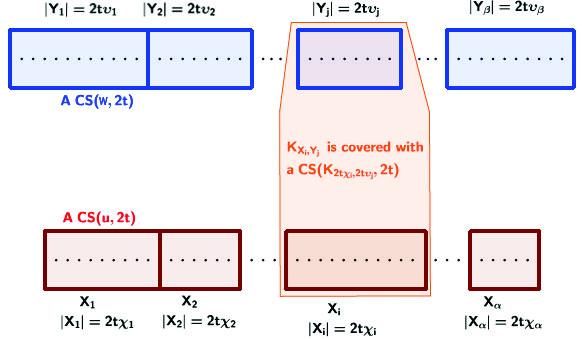}
  \end{figure}

To show the effectiveness of Construction \ref{constr}, we only need to check that any edge of $K_{U,W}$ is contained in a cycle of some $\mathcal{C}_{i,j}$. In fact,
since the edge-sets covered by the $\mathcal{C}_{i,j}$s are disjoint, it follows that there is exactly one cycle containing such an edge.
For any given pair $(x,y)\in U\times W$ we have that $x \in X_i$ and $y\in Y_j$ for suitable $i$ and $j$.
Note now that $\mathcal{C}_{i,j}$ is a cycle system of $K_{X_i, Y_j}$. 
Therefore,  the edge $\{x,y\}$ is contained in a cycle of 
$\mathcal{C}_{i,j}$ and this is enough to show that $\mathcal{F}$ is a $CS(u+w,2t)$.

Our main result will follow straightforwardly by the following three theorems.

\begin{thm}\label{main1}
If there exists an intersecting $CS(K_{2t,2t}, 2t)$, then
there is an unparalleled $CS(v,2t)$ for any $v \equiv 0 \pmod{2t}$ with $v\geq 4t\geq 12$.
\end{thm}
\begin{proof} Let $v = 2tq$ and  proceed by induction on $q$. If $q=2$, then the assertion follows by Theorem \ref{int1}. 
If $q > 2$, let $\mathcal{W}$ be an unparalleled $CS(v-2t,2t)$ which exists  by the induction hypothesis. 
We denote by $W$ the vertex-set of $\mathcal{W}$ and by $\{Y_1, Y_2, \ldots,Y_{q-1}\}$  a partition of $W$ into $2t$-sets. 

Let $\mathcal{U}$ be a $CS(2t,2t)$ and denote by $X$ its  vertex-set.
By assumption, there exists an intersecting $2t$-cycle system $\mathcal{C}^*$ of $K_{X,Y}$ whenever 
$|Y|=2t$.
For any $j\geq 1$ we denote by $\mathcal{C}_j$ a copy  of $\mathcal{C}^*$ that covers the edges of 
$K_{X, Y_j}$; more precisely,
let $f_j$ be a bijection 
between $X \ \cup\ Y$ and $X \ \cup \ Y_j$ fixing $X$ pointwise, and set $\mathcal{C}_j = f_j(\mathcal{C^*})$.
By Construction \ref{constr}, we have that 
$\mathcal{F} = \mathcal{U} \ \cup \ \mathcal{W} \ \cup \ \bigcup_{j} \mathcal{C}_{j}$
 is a $CS(v, 2t)$. 
 
 We claim that $\mathcal{F}$ is parallel class free. Assume for a contradiction that $\mathcal{F}$ has a 
 parallel class $\mathcal{P}$. Note that the vertices of $X$ can be covered either by a unique cycle of 
 $\mathcal{U}$ or by two cycles in $\bigcup_{j} \mathcal{C}_{j}$.  If $\mathcal{P}$ contains a cycle 
 $C\in \mathcal{U}$, then
 $\mathcal{P}\setminus\{C\}$ is a parallel class of $\mathcal{W}$, contradicting the fact that $\mathcal{W}$ 
 has no parallel class. Then $\mathcal{P}$ contains two cycles $C_1, C_2$ of $\bigcup_{j} \mathcal{C}_{j}$. 
 By construction, $C_i=f_{j_i}(C_i^*)$ for some $C_i^* \in \mathcal{C}^*$ and for some integer $j_i$, with  
 $i=1,2$.
 Since the $f_j$'s fix the vertices in $X$, then $V(C_i)\ \cap\ X = V(C_i^*) \ \cap \ X$ for $i=1,2$. 
 By taking into account that the cycles of $\mathcal{C}^*$ satisfy condition \eqref{good}, we have that 
\[
  V(C_1) \ \cap\ V(C_2) \ \cap\ X = V(C_1^*) \ \cap\ V(C_2^*) \ \cap\ X \neq \emptyset.
\]
 This means that $C_1$ and $C_2$ share vertices contradicting the fact that $P$ is a parallel class.
 We can then conclude that $\mathcal{F}$ is unparalleled.
\end{proof}

\begin{thm}\label{main2}
There exists an unparalleled $CS(v,2t)$ for any odd $t> 5$ and for any $v \equiv 0 \pmod{4t}$.
\end{thm}
\begin{proof} Let $v = 4tq$ and  proceed by induction on $q$. If $q=1$, then the assertion follows by  Theorem \ref{int1}. 
If $q > 1$, let $\mathcal{W}$ be an unparalleled $CS(v-4t,2t)$ which exists  by the induction hypothesis. 
We denote by $W$ the vertex-set of $\mathcal{W}$ and by $\{Y_1, Y_2, \ldots,Y_{q-1}\}$  a partition of $W$ into $4t$-sets. 

By Theorem \ref{int1}, there exists an intersecting $CS(4t,2t)$, say $\mathcal{U}$; 
we denote its  vertex-set by $X$. Also, by Theorem \ref{strong2}, there exists an intersecting $2t$-cycle system $\mathcal{C}^*$ of $K_{X,Y}$ whenever $|Y|=4t$.
For any $j\geq 1$ we denote by $\mathcal{C}_j$ a copy  of $\mathcal{C}^*$ that covers the edges of 
$K_{X, Y_j}$; more precisely,
let $f_j$ be a bijection 
between $X \ \cup\ Y$ and $X \ \cup \ Y_j$ fixing $X$ pointwise, and set $\mathcal{C}_j = f_j(\mathcal{C^*})$.
By Construction \ref{constr}, we have that 
$\mathcal{F} = \mathcal{U} \ \cup \ \mathcal{W} \ \cup \ \bigcup_{j} \mathcal{C}_{j}$
 is a $CS(v, 2t)$. 
 
 We claim that $\mathcal{F}$ is unparalleled. Assume by contradiction that $\mathcal{F}$ has a 
 parallel class $\mathcal{P}$. Note that the vertices of $X$ can be covered in $\mathcal{P}$ 
 by two cycles of $\mathcal{U}$, or 
 by four cycles in $\bigcup_{j} \mathcal{C}_{j}$, or 
 by one cycle of $\mathcal{U}$ and two cycles in $\bigcup_{j} \mathcal{C}_{j}$.  Since any two cycles of $\mathcal{U}$ intersect in some vertex, it then follows that
 $\mathcal{P}$ contains at least two cycles $C_1$ and $C_2$ in $\bigcup_{j} \mathcal{C}_{j}$. However,  
 as in the proof of Theorem \ref{main1}, it is not difficult to show that
 $C_1$ and $C_2$ share vertices contradicting the fact that $\mathcal{P}$ is a parallel class.
 The assertion is therefore proven.
\end{proof}

\begin{thm}\label{main3}
There exists an unparalleled $CS(v,2t)$ for any odd $t>5$ and 
for any $v \equiv 2t \pmod{4t}$ with $v\geq 6t$.
\end{thm}
\begin{proof} 
Let $v = 4tq+2t$. If $q=1$, then the assertion follows by  
Proposition \ref{int2}. 

If $q > 1$, then $v-6t\equiv 0\pmod{4t}$ and by Theorem \ref{main2} there is an unparalleled $CS(v-6t,2t)$,
say $\mathcal{W}$. 
We denote by $W$ the vertex-set of $\mathcal{W}$ and by $\{Y_1, Y_2, \ldots,Y_{q-1}\}$  a partition of $W$ into $4t$-sets. Also, by 
Proposition \ref{int2}, there exists an intersecting $CS(6t,2t)$, say $\mathcal{U}$; 
we denote its  vertex-set by $U$ and partition it into subsets $X_1$ and $X_2$ of size $4t$ and $2t$,
respectively. 
By Theorem \ref{strong2}, there exists an intersecting 
$2t$-cycle system $\mathcal{C}^*$ of $K_{X_1,Y}$ whenever $|Y|=4t$.
For any $j\geq 1$ we denote by $\mathcal{C}_{1,j}$ a copy  of $\mathcal{C}^*$ that covers the edges of 
$K_{X_1, Y_j}$; more precisely,
let $f_j$ be a bijection 
between $X \ \cup\ Y$ and $X \ \cup \ Y_j$ fixing $X$ pointwise, 
and set $\mathcal{C}_{1,j} = f_j(\mathcal{C^*})$.
Finally, let $\mathcal{C}_{2,j}$ be a $2t$-cycle system of $K_{X_2,Y_j}$ for any $j$ (see \cite{So81}).
By Construction \ref{constr}, we have that 
$\mathcal{F} = \mathcal{U} \ \cup \ \mathcal{W} \ \cup \ \bigcup_{i,j} \mathcal{C}_{i,j}$
 is a $CS(v, 2t)$. 
 
 We claim that $\mathcal{F}$ is unparalleled. Assume by contradiction that $\mathcal{F}$ has a 
 parallel class $\mathcal{P}$. Since $\mathcal{U}$ is intersecting, any two cycles of $\mathcal{U}$ share some  
 vertex;
 therefore, $\mathcal{P}$ contains at most one cycle $C_0$ of $\mathcal{U}$. Since $|X_1|=4t$, 
 the vertices of $X_1$ not lying in $C_0$ are covered by at least two cycles $C_1$ and $C_2$ in 
 $\bigcup_j \mathcal{C}_{1,j}$.
 As in the proof of Theorem \ref{main1}, it not difficult to show that
 $C_1$ and $C_2$ share vertices contradicting the fact that $\mathcal{P}$ is a parallel class.
\end{proof}

We are now able to prove the main theorem of this paper.\\

\noindent
{\bf Theorem \ref{main}.}
\emph{There exists an unparalleled $CS(v,2t)$ with $v \equiv 0 \pmod{2t}$ if and only if $v>2t>2$. }

\begin{proof} The case $t=2$ is proven in Theorem \ref{four}.
For $t=3,5$ or $t>2$ even, the result follows from Theorems \ref{strong1}, \ref{strong3}, \ref{strong5}, and \ref{main1}.
For an odd $t>5$ and $v\equiv 0 \pmod{4t}$, the assertion is proven in
Theorem \ref{main2}.
For an odd $t> 5$ and $v\equiv 2t\pmod{4t}$, the result is proven in
Theorem \ref{main3}.
\end{proof}

\section{Conclusions}
It appears as though most designs are non-resolvable but nonetheless have parallel classes,
whereas unparalleled designs are rarer and more difficult to construct.
There has been work showing that for any admissible order 
there exists a Steiner triple system which is not resolvable \cite{LiRees05}. This begs the question of whether there exists such designs which contain no parallel classes. 
This has been considered in \cite{BrHo15} but the problem still remains open.

In this paper, we have completely solved the case of even cycle systems of $K_v-I$. Results for cycle systems whose cycles are of odd order remain elusive.
Of course, these kind of questions can be considered for many other kinds of decompositions.
 The projective planes and, more generally, symmetric designs are known examples of intersecting systems. However, the problem in the case of non-symmetric BIBDs appears to be intractable.

It is classically known that there exist Latin squares which contain no transversal
\cite{DeKe74}. More generally, work has been done on bachelor and monogamous Latin squares
which can be seen as a kind of non-resolvable designs \cite{DaWaWe11, Ev06, WaWe06}.
Of course, there are many other decompositions into graphs such as paths or small trees 
which may be amenable to the methods of this paper.

\end{document}